\documentclass[12pt]{article}

\usepackage{verbatim}
\usepackage{setspace}
\usepackage{sectsty}
\subsectionfont{\normalsize}
\makeatletter
\renewcommand\section{\@startsection{section}{1}{\z@}%
                                  {-2.0ex \@plus -1ex \@minus -.2ex}%
                                  {2.0ex \@plus.2ex}%
                                  {\normalfont\normalsize\bfseries}}
\makeatother
\usepackage{fullpage}
\usepackage{url}
\usepackage{amsmath}
\usepackage{mathdots}
\usepackage[pdftex]{graphicx}
\newcommand{\erv}{2^{2^{4k-\lg(k-2)}}}

\newcommand{\erveasy}{2^{2^{4k}}}
\newcommand{\cfsv}{2^{B(k-1)^{1/2}2^{2k}}}

\newcommand{\cfsveasy}{2^{2^{(2+o(1))k}}}
\newcommand{\squash}{{\rm squash}}
\newcommand{\TOW}{{\rm TOW}}
\newcommand{\rtkf}{R(3,k,4)}
\newcommand{\rtkt}{R(3,k)}
\newcommand{\rtkmt}{R(3,k-1)}
\newcommand{\rtwkt}{R(2,k)}
\newcommand{\rtwkmt}{R(2,k-1)}
\newcommand{\rtwkc}{R(2,k,c)}

\newcommand{\rokt}{R(1,k)}
\newcommand{\roktc}{R(1,k,c)}
\newcommand{\rtkc}{R(3,k,c)}
\newcommand{\raktreal}{R(a,k,2)}
\newcommand{\rakt}{R(a,k)}

\newcommand{\ramstuff}{R(a-1,2\uparrow^{a-1}(2k-(i+1)))}
\newcommand{\ramkt}{R(a-1,k)}
\newcommand{\ramkmt}{R(a-1,k-1)}
\newcommand{\ramkmc}{R(a-1,k-1,c)}
\newcommand{\rammkmc}{R(a-2,k-1,c)}

\newcommand{\rammkmt}{R(a-2,k-1)}

\newcommand{\rakc}{R(a,k,c)}

\newcommand{\rfokt}{R(4,k)}
\newcommand{\rfoktc}{R(4,k,c)}

\newcommand{\rfikt}{R(5,k)}
\newcommand{\rfiktc}{R(5,k,c)}

\newcommand{\Erdos}{Erd\H{o}s }
\newcommand{\Erdosns}{Erd\H{o}s}

\newcommand{\ang}[1]{\langle#1\rangle}

\newcommand{\kstar}{{\textstyle *}}

\newcommand{\nat}{{\sf N}}

\newcommand{\xvec}[1]{\ifcase 3{#1} {\ang {x_1,x_2,x_3} } \else 
\ifcase 4{#1} {\ang{x_1,x_2,x_3,x_4}} \else {\ang {x_1,\ldots,x_{#1}}}\fi\fi}
\newcommand{\yvec}[1]{\ifcase 3{#1} {\ang {y_1,y_2,y_3} } \else 
\ifcase 4{#1} {\ang{y_1,y_2,y_3,y_4}} \else {\ang {y_1,\ldots,y_{#1}}}\fi\fi}
\newcommand{\zvec}[1]{\ifcase 3{#1} {\ang {z_1,z_2,z_3} } \else 
\ifcase 4{#1} {\ang{z_1,z_2,z_3,z_4}} \else {\ang {z_1,\ldots,z_{#1}}}\fi\fi}
\newcommand{\vecc}[2]{\ifcase 3{#2} {\ang { {#1}_1,{#1}_2,{#1}_3 } } \else
\ifcase 4{#1} {\ang { {#1}_1,{#1}_2,{#1}_3,{#1}_{4} } }
\else {\ang { {#1}_1,\ldots,{#1}_{#2}}}\fi\fi}
\newcommand{\veccd}[3]{\ifcase 3{#2} {\ang { {#1}_{{#3}1},{#1}_{{#3}2},{#1}_{{#3}3} } } \else
\ifcase 4{#1} {\ang { {#1}_{{#3}1},{#1}_{{#3}2},{#1}_{#3}3},{#1}_{{#3}4} }
\else {\ang { {#1}_{{#3}1},\ldots,{#1}_{{#3}{#2}}}}\fi\fi}
\newcommand{\veccz}[2]{\ifcase 3{#2} {\ang { {#1}_0,{#1}_2,{#1}_3 } } \else
\ifcase 4{#1} {\ang { {#1}_0,{#1}_2,{#1}_3,{#1}_{4} } }
\else {\ang { {#1}_0,\ldots,{#1}_{#2}}}\fi\fi}
\newcommand{\xve}[1]{\ifcase 3{#1} {x_1,x_2,x_3} \else 
\ifcase 4{#1} {x_1,x_2,x_3,x_4} \else {x_1,\ldots,x_{#1}}\fi\fi}
\newcommand{\yve}[1]{\ifcase 3{#1} {y_1,y_2,y_3} \else 
\ifcase 4{#1} {y_1,y_2,y_3,y_4} \else {y_1,\ldots,y_{#1}}\fi\fi}
\newcommand{\zve}[1]{\ifcase 3{#1} {z_1,z_2,z_3} \else 
\ifcase 4{#1} {z_1,z_2,z_3,z_4} \else {z_1,\ldots,z_{#1}}\fi\fi}
\newcommand{\ve}[2]{\ifcase 3#2 {{#1}_1,{#1}_2,{#1}_3} \else
\ifcase 4#2 {{#1}_1,{#1}_2,{#1}_3,{#1}_{4}}
\else {{#1}_1,\ldots,{#1}_{#2}}\fi\fi}
\newcommand{\ved}[3]{\ifcase 3#2 {{#1}_{{#3}1},{#1}_{{#3}2},{#1}_{{#3}3}} \else
\ifcase 4#2 {{#1}_{{#3}1},{#1}_{{#3}2},{#1}_{{#3}3},{#1}_{{#3}4}}
\else {{#1}_{{#3}1},\ldots,{#1}_{{#3}{#2}}}\fi\fi}
\newcommand{\fuve}[3]{
\ifcase 3#2
{{#3}({#1}_1),{#3}({#1}_2,{#3}({#1}_3)} \else
\ifcase 4#2
{{#3}({#1}_1),{#3}({#1}_2),{#3}({#1}_3),{#3}({#1}_4)}
\else
{{#3}({#1}_1),\ldots,{#3}({#1}_{#2})}\fi\fi}
\newcommand{\setmathchar}[1]{\ifmmode#1\else$#1$\fi}
\newcommand{\vlist}[2]{%
	\setmathchar{%
		\compound#2\one{#2}\two
		\ifcompound
			({#1}_1,\ldots,{#1}_{#2})
		\else
			\ifcat N#2
				({#1}_1,\ldots,{#1}_{#2})
			\else
				\ifcase#2
					({#1}_0)\or
					({#1}_1)\or
					({#1}_1,{#1}_2)\or 
					({#1}_1,{#1}_2,{#1}_3)\or
					({#1}_1,{#1}_2,{#1}_3,{#1}_4)\else 
					({#1}_1,\ldots,{#1}_{#2})
				\fi
			\fi
		\fi}}

\newif\ifcompound
\def\compound#1\one#2\two{%
	\def\one{#1}
	\def\two{#2}
	\if\one\two
		\compoundfalse
	\else
		\compoundtrue
	\fi}

\newcommand{\xwe}[1]{\ifcase 3{#1} {x_1\wedge x_2\wedge x_3} \else 
\ifcase 4{#1} {x_1\wedge x_2\wedge x_3\wedge x_4} \else {x_1\wedge \cdots \wedge
x_{#1}}\fi\fi}
\newcommand{\we}[2]{\ifcase 3#2 {\ang { {#1}_1\wedge {#1}_2\wedge {#1}_3 } } \else
\ifcase 4{#1} {\ang { {#1}_1\wedge {#1}_2\wedge {#1}_3\wedge {#1}_{4} } }
\else {\ang { {#1}_1\wedge \cdots\wedge {#1}_{#2}}}\fi\fi}

\newcommand{\into}{\rightarrow}

\newcommand{\es}{\emptyset}

\newcommand{\floor}[1]{\left\lfloor{#1}\right\rfloor}

\newcommand{\union}{\cup}

\newcommand{\s}[1]{\s_{#1}}

\newcommand{\monus}{\;\raise.5ex\hbox{{${\buildrel
    \ldotp\over{\hbox to 6pt{\hrulefill}}}$}}\;}

\newcounter{savenumi}

\newtheorem{theoremfoo}{Theorem}[section] 
\newenvironment{theorem}{\pagebreak[1]\begin{theoremfoo}}{\end{theoremfoo}}

\newtheorem{lemmafoo}[theoremfoo]{Lemma}
\newenvironment{lemma}{\pagebreak[1]\begin{lemmafoo}}{\end{lemmafoo}}
\newtheorem{conjecturefoo}[theoremfoo]{Conjecture}

\newtheorem{conventionfoo}[theoremfoo]{Convention}
\newenvironment{convention}{\pagebreak[1]\begin{conventionfoo}\rm}{\end{conventionfoo}}

\newtheorem{porismfoo}[theoremfoo]{Porism}

\newtheorem{gamefoo}[theoremfoo]{Game}

\newtheorem{corollaryfoo}[theoremfoo]{Corollary}
\newenvironment{corollary}{\pagebreak[1]\begin{corollaryfoo}}{\end{corollaryfoo}}

\newtheorem{openfoo}[theoremfoo]{Open Problem}

\newtheorem{exercisefoo}{Exercise}

\newcommand{\fig}[1] 
{
 \begin{figure}
 \begin{center}
 \input{#1}
 \end{center}
 \end{figure}
}

\newtheorem{potanafoo}[theoremfoo]{Potential Analogue}

\newtheorem{notefoo}[theoremfoo]{Note}
\newenvironment{note}{\pagebreak[1]\begin{notefoo}\rm}{\end{notefoo}}

\newtheorem{notabenefoo}[theoremfoo]{Nota Bene}

\newtheorem{nttn}[theoremfoo]{Notation}
\newenvironment{notation}{\pagebreak[1]\begin{nttn}\rm}{\end{nttn}}

\newtheorem{empttn}[theoremfoo]{Empirical Note}

\newtheorem{examfoo}[theoremfoo]{Example}
\newenvironment{example}{\pagebreak[1]\begin{examfoo}\rm}{\end{examfoo}}

\newtheorem{dfntn}[theoremfoo]{Def}
\newenvironment{definition}{\pagebreak[1]\begin{dfntn}\rm}{\end{dfntn}}

\newtheorem{propositionfoo}[theoremfoo]{Proposition}

\newenvironment{proof}
    {\pagebreak[1]{\narrower\noindent {\bf Proof:\quad\nopagebreak}}}{\QED}

\newcommand{\yyskip}{\penalty-50\vskip 5pt plus 3pt minus 2pt}
\newcommand{\blackslug}{\hbox{\hskip 1pt
        \vrule width 4pt height 8pt depth 1.5pt\hskip 1pt}}
\newcommand{\QED}{{\penalty10000\parindent 0pt\penalty10000
        \hskip 8 pt\nolinebreak\blackslug\hfill\lower 8.5pt\null}
        \par\yyskip\pagebreak[1]}

\newcommand{\BBB}{{\penalty10000\parindent 0pt\penalty10000
        \hskip 8 pt\nolinebreak\hbox{\ }\hfill\lower 8.5pt\null}
        \par\yyskip\pagebreak[1]}

\newtheorem{factfoo}[theoremfoo]{Fact}

\newenvironment{block}{\begin{list}{\hbox{}}{\leftmargin 1em
    \itemindent -1em \topsep 0pt \itemsep 0pt \partopsep 0pt}}{\end{list}}

\begin{document}

\newcommand{\KN}{K_{\nat}}
\newcommand{\NRE}{\hbox{NUM-RED-EDGES\ }}
\newcommand{\NBE}{\hbox{NUM-BLUE-EDGES\ }}
\newcommand{\RED}{\hbox{RED\ }}
\newcommand{\BLUE}{\hbox{BLUE\ }}
\newcommand{\REDns}{\hbox{RED}}
\newcommand{\BLUEns}{\hbox{BLUE}}

\title{Three Proofs of the Hypergraph Ramsey Theorem (An Exposition)}

\author{
{William Gasarch}
\thanks{University of Maryland at College Park,
Department of Computer Science,
        College Park, MD\ \ 20742.
\texttt{gasarch@cs.umd.edu}
}
\\ {\small Univ. of MD at College Park}
\and
{Andy Parrish}
\thanks{University of California at San Diego,
Department of Mathematics,
9500 Gillman Dr,
La Jolla, CA\ 92093.
\texttt{atparrish@ucsd.edu}
}
\\ {\small Univ. of CA at San Diego}
\and
{Sandow Sinai}
\thanks{Poolesville High School,
Poolesville, MD, 20837
\texttt{sandow.sinai@gmail.com}
}
\\ {\small Poolesville High School}
}

\date{}

\maketitle

\begin{abstract}
Ramsey, \Erdosns-Rado, and Conlon-Fox-Sudakov have given
proofs of the 3-hypergraph Ramsey Theorem with better and
better upper bounds on the 3-hypergraph Ramsey numbers.
Ramsey and \Erdosns-Rado also prove the $a$-hypergraph Ramsey Theorem.
Conlon-Fox-Sudakov note that their upper bounds on the 3-hypergraph Ramsey Numbers,
together with a recurrence of \Erdosns-Rado (which was the key to the \Erdosns-Rado proof),
yield improved  bounds on the $a$-hypergraph Ramsey numbers.
We present all of these proofs and state explicit bounds for the 2-color case and the $c$-color case.
We give a more detailed analysis of the construction of Conlon-Fox-Sudakov
and hence obtain a slightly better bound.
\end{abstract}

\section{Introduction}\label{se:intro}

The 3-hypergraph Ramsey numbers $\rtkt$ 
were first shown to exist by Ramsey~\cite{Ramsey}.
His upper bounds on them were enormous.
\Erdosns-Rado~\cite{ErdosRado} obtained much better bounds,
namely $\rtkt \le \erveasy$.
Recently Conlon-Fox-Sudakov~\cite{conlonfoxsud} have
obtained $\rtkt \le \cfsveasy$.
We present all three proofs. 
For the Conlon-Fox-Sudakov proof 
we give a more detailed analysis that required a nontrivial lemma, and hence
we obtain slightly better bounds.
Before starting the second and third proofs we will discuss
why they improve the prior ones.

We also present extensions of all three proofs to the $a$-hypergraph case.
The first two are known proofs and bounds.
The \Erdosns-Rado proof gives a recurrence to obtain $a$-hypergraph
Ramsey Numbers from $(a-1)$-hypergraph Ramsey Numbers.
As Conlon-Fox-Sudakov note, this recurrence together with their improved
bound on $\rtkt$, yield better upper bounds on the $a$-hypergraph Ramsey Numbers.
Can the Conlon-Fox-Sudakov method itself be extended to a proof of the $a$-hypergraph Ramsey Theorem?
It can; however (alas), this does not seem to lead to better upper bounds.
We include this proof in the appendix in the hope that someone may 
improve either the construction or the analysis to obtain better
bounds on the $a$-hypergraph Ramsey Numbers.

For all of the proofs, the extension to $c$ colors is routine.
We present the results as notes; however, we leave the proofs as easy exercises for the reader.

\section{Notation and Ramsey's Theorem}

\begin{definition}
Let $X$ be a set and $a\in\nat$. Then $\binom{X}{a}$ is the set of
all subsets of $X$ of size $a$.
\end{definition}

\begin{definition}
Let $a,n\in\nat$.
The {\it complete $a$-hypergraph on $n$ vertices}, denoted $K_n^a$,
is the hypergraph  with vertex set 
$V=[n]$ and edge set $E=\binom{[n]}{a}$
\end{definition}

\begin{notation}
In this paper a {\it coloring of a graph or hypergraph} always means a coloring of the {\it edges}.
We will abbreviate $COL(\{x_1,\ldots,x_a\})$ by $COL(x_1,\ldots,x_a)$.
We will refer to a $c$-coloring of the edges of the complete hypergraph $K_n^a$ as
a $c$-coloring of $\binom{[n]}{a}$.
\end{notation}

\begin{definition}\label{de:homog}
Let $a\ge 1$.
Let $COL$ be a $c$-coloring of $\binom{[n]}{a}$. 
A set of vertices $H$ is {\it $a$-homogeneous for $COL$}
if every edge in $\binom{H}{a}$ is the same color.
We will drop the {\it for $COL$} when it is understood.
We will drop the $a$  when it is understood.
\end{definition}

\begin{convention}
When talking about 2-colorings will often denote the colors by $\RED$ and $\BLUEns$.
\end{convention}

\begin{note}
In Definition~\ref{de:homog} we allow $a=1$.
Note that a $c$-coloring of $\binom{[n]}{1}$ is just
a coloring of the numbers in $[n]$.
A homogenous subset $H$ is a subset of points that are all colored the same.
Note that in this case the edges are 1-subsets of the points and hence
are identified with the points.
\end{note}

\begin{definition}
Let $a,c,k\in\nat$.
Let $\rakc$ be the least $n$ such that, for all $c$-colorings
of $\binom{[n]}{a}$ there exists an $a$-homogeneous set $H\in \binom{[n]}{k}$.
We denote $\raktreal$ by $\rakt$.
We have not shown that $\rakc$ exists; however, we will.
\end{definition}

We state Ramsey's theorem 
for 1-hypergraphs (which is trivial) and 
for 2-hypergraphs (just graphs).
The 2-hypergraph case (and the $a$-hypergraph case) 
is due to Ramsey~\cite{Ramsey} (see also~\cite{ramseynotes,GRS,RamseyInts}).
The bound we give on $\rtwkt$ seems to be folklore (see~\cite{GRS}).

\begin{definition}
The expression $\omega(1)$ means a function that goes to infinity monotonically.
For example, $\floor{\lg\lg n}$.
\end{definition}

The following are well known.

\begin{theorem}\label{th:ramsey2}~
Let $k\in\nat$ and $c\ge 2$.
\begin{enumerate}
\item
$\rokt = 2k-1$.
\item
$\roktc = ck-c+1$.
\item
$\rtwkt\le \binom{2k-2}{k-1} \le 2^{2k -0.5\lg(k-1) -\Omega(1)}.$
\item
$\rtwkc \le \frac{(c(k-1))!}{(k-1)!)^c} \le c^{ck - 0.5\log_c(k-1)+O(c)}.$
\item
For all $n$, for every 2-coloring of $\binom{[n]}{2}$, there exists a 2-homogenous set 
$H$ of size at least $\frac{1}{2}\lg n + \omega(1) $.
(This follows from Part 3 easily. In fact, 
all you need is $\rtwkt\le 2^{2k-\Omega(1)}$.)
\end{enumerate}
\end{theorem}

\begin{note}\label{no:ramsey2}
Theorem~\ref{th:ramsey2}.2 has an elementary proof. A more sophisticated proof,
by David Conlon~\cite{ramseyupper} yields
$\rtwkt \le  k^{-E\frac{\log k}{\log \log k}}\binom{2k}{k}$, where $E$ is some constant.
A simple probabilistic argument shows that $\rtwkt\ge (1+o(1))\frac{1}{e\sqrt{2}}k2^{k/2}$.
A more sophisticated argument by Spencer~\cite{ramseylower} (see~\cite{GRS}),
that uses the Lovasz Local Lemma, shows
$\rtwkt\ge (1+o(1))\frac{\sqrt{2}}{e}k2^{k/2}$.
\end{note}

We state Ramsey's theorem on $a$-hypergraphs~\cite{Ramsey} 
(see also~\cite{GRS,RamseyInts}).

\begin{theorem}\label{th:ramseyac}
Let $a,k,c\in\nat$.
For all $k\in \nat$, $\rakc$ exists.
\end{theorem}

\section{Summary of Results}

We will need both the tower function and Knuth's arrow notation
to state the results.

\begin{notation}
\begin{equation*}
c \uparrow^a k=
\begin{cases}
ck		    & \text{ if $a=0$, } \\
c^k,                & \text{ if $a=1$,} \\
1,                  & \text{ if $k=0$,} \\
c \uparrow^{a-1}(c \uparrow^a (k-1)), & \text{ otherwise}.\cr
\end{cases}
\end{equation*}
\end{notation}

\begin{definition}
We define $\TOW$ which takes on a variable number of
arguments.
\begin{enumerate}
\item
$\TOW_{c}(b)= c^b$.
\item
$\TOW_{c}(b_1,\ldots,b_L) = c^{b_1 \TOW_c(b_2,\ldots,b_L)}$.
\end{enumerate}
When $c$ is not stated it is assumed to be 2.
\end{definition}

\begin{example}~
\begin{enumerate}
\item
$\TOW(2k)=2^{2k}$.
\item
$\TOW(1,4k)= 2^{2^{4k}}$.
\item
$\TOW(1)=2$, $\TOW(1,1)=2^2$, $\TOW(1,1,1)=2^{2^2}$.
\end{enumerate}
\end{example}

The list below contains both who proved what bounds
and the results we will prove in this paper.

\begin{enumerate}
\item
Ramsey's proof~\cite{Ramsey} yields:
\begin{enumerate}
\item
$\rtkt \le 2\uparrow^2(2k-1)=\TOW(1,\ldots,1)$ where the number of 1's is $2k-1$.
\item
$\rakt\le 2\uparrow^{a-1}(2k-1)$.
\end{enumerate}
\item
The \Erdosns-Rado~\cite{ErdosRado} proof yields:
\begin{enumerate}
\item
$\rtkt\le \erv$.
\item
$\rakt \le  2^{\binom{\ramkmt+1}{a-1}}+a-2.$
\item
Using the recurrence they obtain the following:
For all $a\ge 4$, 
$\rakt \le \TOW(1,a-1,a-2,\ldots,3,4k-\lg (k-a+1)-4(a-3)))$.
\end{enumerate}
\item
The Conlon-Fox-Sudakov~\cite{conlonfoxsud} proof yields:
\begin{enumerate}
\item
$\rtkt \le \cfsv$ where $B=\bigl(\frac{e}{\sqrt{2\pi}}\bigr)^3\sim 1.28$.
\item
If you combine this with the recurrence obtained by \Erdosns-Rado then one obtains:
\begin{enumerate}
\item
$\rtkt \le \TOW(B(k-1)^{1/2},2^{2k})$.
\item
$\rfokt \le \TOW(1,3B(k-2)^{1/2},2^{2k-2})$.
\item
$\rfikt \le \TOW(1,4,3B(k-3)^{1/2},2^{2k-4})$.
\item
For all $a\ge 6$, for almost all $k$, 
$$\rakt \le \TOW(1,a-1,a-2,\ldots,4,3B(k-a+2)^{1/2},2^{2k-2a+6}).$$
\end{enumerate}
\end{enumerate}
\item
The Appendix contains an alternative proof of the $a$-hypergraph Ramsey Theorem
based on the ideas of Conlon-Fox-Sudakov. Since it does not yield better bounds
we do not state the bounds here.
\end{enumerate}

\begin{notation}
PHP stands for Pigeon Hole Principle.
\end{notation}

We will need the following lemma whose easy proof we 
leave to the reader.

\begin{lemma}\label{le:tow}
For all $b,b_1,\ldots,b_L\in\nat$ the following hold.
\begin{enumerate}
\item
$\TOW(b_1,\ldots,b_i,b_{i+1},b_{i+2}\ldots,b_L) \le \TOW(b_1,\ldots,1,b_{i+1}+\lg(b_i),b_{i+2},\ldots,b_L).$
\item
$\TOW(b_1,\ldots,b_L)^b = \TOW(bb_1,b_2,\ldots,b_L)$.
\item
$(1+\delta)\TOW(b_1,\ldots,b_L) \le
\TOW(b_1,b_2,\ldots,b_L+\delta)$.
\item
$(1+\delta)\TOW(b_1,\ldots,b_L)^b \le \TOW(bb_1,b_2,\ldots,b_L+\delta)$. (This follows from 1 and 2.)
\item
$2^{\TOW(b_1,\ldots,b_L)}= \TOW(1,b_1,\ldots,b_L)$.
\item
$2^{(1+\delta)\TOW(b_1,\ldots,b_L)^b}\le \TOW(1,bb_1,b_2,\ldots,b_L+\delta)$.
(This follows from 4 and 5.)
\item
$\lg^{(c)}(\TOW(1,\ldots,1))=1$ (there are $c$ 1's).
\end{enumerate}
\end{lemma}

\section{Ramsey's Proof}\label{se:ramsey}

\begin{theorem}\label{th:ramsey3}
For almost $k$ $\rtkt\le 2\uparrow^2(2k-1)=\TOW(1,\ldots,1)$ where there are $2k-1$ 1's.
\end{theorem}

\begin{proof}

Let $n$ be a number to be determined.
Let $COL$ be a 2-coloring of $\binom{[n]}{3}$.
We define a sequence of vertices, 
$$x_1,x_2,\ldots,x_{2k-1}.$$ 

Here is the basic idea: Let $x_1=1$.
This induces the following coloring of $\binom{[n]-\{1\}}{2}$:
$$COL^*(x,y) = COL(x_1,x,y).$$
By Theorem \ref{th:ramsey2}
there exists a 2-homogeneous set for $COL^*$ of size $\frac{1}{2}\lg n +\omega(1)$.
Keep that 2-homogeneous set and ignore the remaining points.
Let $x_2$ be the least vertex that has been kept (bigger than $x_1$).
Repeat the process.

We describe the construction formally.

\bigskip

\noindent
{\bf CONSTRUCTION}

$V_0=[n]$

Assume $1\le i \le 2k-1$ and that $V_{i-1}$, $x_1,x_2,\ldots,x_{i-1}$, 
$c_1,\ldots,c_{i-1}$ are all defined. 
We define $x_i$, $COL^*$, $V_i$, and $c_i$:

\[
\begin{array}{rl}
x_i =& \hbox{ the least number in $V_{i-1}$} \cr
V_{i} =& V_{i-1} - \{x_i\} \hbox{ (We will change this set without changing its name.)} \cr
COL^*(x,y)=&COL(x_{i},x,y) \hbox{ for all }\{x,y\} \in \binom{V_{i}}{2}\cr
V_{i} = & \hbox{ the largest 2-homogeneous set for $COL^*$}\cr
c_{i} = & \hbox{ the color of $V_{i}$}\cr
\end{array}
\]

\noindent
KEY: for all $y,z\in V_{i}$, $COL(x_i,y,z)=c_i$.

\noindent
{\bf END OF CONSTRUCTION}

\bigskip

When we derive upper bounds on $n$ we will show that the construction can be carried out for $2k-1$ stages.  For now assume the construction ends.

We have vertices

$$x_1,x_2,\ldots,x_{2k-1}$$

and associated colors

$$c_1, c_2,\ldots,c_{2k-1}.$$

There are only two colors, hence, by PHP,
there exists 
$i_1,\ldots,i_{k}$ such that $i_1<\cdots < i_{k}$ and
$$c_{i_1} = c_{i_2} = \cdots = c_{i_{k}}$$
We take this color to be $\REDns$.
We show that 

$$H=\{x_{i_1}, x_{i_2},  \ldots,  x_{i_{k}}\}.$$

\noindent
is $3$-homogenous for $COL$.
For notational convenience we show 
that $COL(x_{i_1},x_{i_2},x_{i_3})=\REDns$.
The proof for any $3$-set of $H$ is similar.
By the definition of $c_{i_1}$
$(\forall A \in \binom{V_{i_1}-\{x_{i_1}\}}{2})[COL(A\cup \{ x_{i_1}) \})=c_i]$
In particular
$$COL(x_{i_1},x_{i_2},x_{i_3})=c_{i_1}=\REDns.$$

We now see how large $n$ must be so that the construction can be carried out.
By Theorem \ref{th:ramsey2}, if $k$ is large,
at every iteration $V_i$ gets reduced by a logarithm, cut in half, and then an $\omega(1)$
is added.  Using this it is easy to show that, for almost all $k$,

$$|V_j| \ge \frac{1}{2}(\lg^{(j)} n)+\omega(1).$$

We want to run this iteration $2k-1$ times 
Hence we need

$$|V_{2k-1}| \ge \frac{1}{2}\log_2^{(2k-1)} n +\omega(1) \ge 1.$$

We can take $n=\TOW(1,\ldots,1)$ where 1 appears $2k-1$ times, and
use Lemma~\ref{le:tow}.
\end{proof}

\begin{note}
The proof of Theorem~\ref{th:ramsey3} generalizes to $c$-colors to yield
$$\rtkc\le c\uparrow^{2}(ck-c+1)=\TOW_c(1,\ldots,1)$$
where the number of 1's is $ck-c+1$.
\end{note}

We now prove Ramsey's Theorem for $a$-hypergraphs.

\begin{theorem}\label{th:ramseya}~
For all $a\ge 1$, for all $k\ge 1$, 
$\rakt\le 2\uparrow^{a-1}(2k-1)$.
\end{theorem}

\begin{proof}

We prove this by induction on $a$. Note that when we
have the theorem for $a$ we have it for $a$ and for {\it all} $k\ge 1$.

\noindent
{\bf Base Case:}
If $a=1$ then, for all $k\ge 1$,  $\rokt=2k-1 \le 2\uparrow^0(2k-1) = 4k-2$.

\noindent
{\bf Induction Step:}
We assume that, for all $k$, $\ramkt \le 2\uparrow^{a-2}(2k-1)$.

Let $k\ge 1$.
Let $n$ be a number to be determined later.
Let $COL$ be a 2-coloring of $\binom{[n]}{a}$.
We show that there is
an $a$-homogenous set for $COL$ of size $k$.

\noindent
{\bf CONSTRUCTION}

$V_0=]n]$.

Assume $1\le i \le 2k-1$ and that $V_{i-1}$, $x_1,x_2,\ldots,x_{i-1}$, 
$c_1,\ldots,c_{i-1}$ are all defined. 
We define $x_i$, $COL^*$, $V_i$, and $c_i$:

\[
\begin{array}{rl}
x_i =& \hbox{ the least number in $V_{i-1}$} \cr
V_{i} =& V_{i-1} - \{x_i\} \hbox{ (We will change this set without changing its name.)} \cr
COL^*(A)=&COL(x_{i}\cup A) \hbox{ for all } A \in \binom{V_i}{a-1})\cr
V_{i} = & \hbox{ the largest $a-1$-homogeneous set for $COL^*$}\cr
c_{i} = & \hbox{ the color of $V_{i}$}\cr
\end{array}
\]

\noindent
KEY: For all $1\le i\le 2k-1$, 
$(\forall A \in \binom{V_{i}}{a-1})[COL(A\cup x_{i})=c_i]$

\noindent
{\bf END OF CONSTRUCTION}

When we derive upper bounds on $n$ we will show that the construction can be carried out for $2k-1$ stages.  For now assume the construction ends.

We have vertices

$$x_1,x_2,\ldots,x_{2k-1}$$

and associated colors

$$c_1, c_2,\ldots,c_{2k-1}.$$

There are only two colors, hence, by PHP,
there exists 
$i_1,\ldots,i_{k}$ such that $i_1<\cdots < i_{k}$ and
$$c_{i_1} = c_{i_2} = \cdots = c_{i_{k}}$$
We take this color to be $\REDns$.
We show that 

$$H=\{x_{i_1}, x_{i_2},  \ldots,  x_{i_{k}}\}.$$

\noindent
is $a$-homogenous for $COL$.
For notational convenience we show 
that $COL(x_{i_1},\ldots,x_{i_a})=\REDns$.
The proof for any $a$-set of $H$ is similar.
By the definition of $c_{i_1}$
$(\forall A \in \binom{V_{i_1}}{a-1})[COL(A\cup x_{i_1})=c_i]$
In particular
$$COL(x_{i_1},\ldots,x_{i_a})=c_{i_1}=\REDns.$$

We show that if $n= 2\uparrow^{a-1} (2k-1)$ then the
construction can be carried out for $2k-1$ stages.

\noindent
{\bf Claim 1:} For all $0\le i\le 2k-1$, $|V_i|\ge 2\uparrow^{a-1}(2k-(i+1))$.

\noindent
{\bf Proof of Claim 1:}
We prove this claim by induction on $i$.
For the base case note that
$$|V_0|=n=2\uparrow^{a-1} (2k-1).$$

Assume $|V_{i-1}|\ge 2\uparrow^{a-1}(2k-i)$.
By the definition of the uparrow function and by the inductive
hypothesis of the theorem,
$$|V_{i-1}|\ge
2\uparrow^{a-1}(2k-i)=
2\uparrow^{a-2}(2\uparrow^{a-1}(2k-(i+1)))
\ge
\ramstuff.
$$

By the construction $V_i$ is the result of applying 
the $(a-1)$-ary Ramsey Theorem to a 2-coloring of $\binom{V_{i-1}}{a}$.
Hence $|V_i| \ge 2\uparrow^{a-1}(2k-(i+1)).$

\noindent
{\bf End of Proof of Claim 1}

By Claim~1 if $n=2\uparrow^{a-1}(2k-1)$ then  
the construction can be carried out for $2k-1$ stages.
Hence $\rakt \le 2\uparrow^{a-1}(2k-1)$.
\end{proof}


The proof of Theorem~\ref{th:ramsey3} is actually an $\omega^2$-induction
that is similar in structure to the original proof of van der Warden's theorem~\cite{VDWbook,GRS,VDW}.

\begin{note}
The proof of Theorem~\ref{th:ramseya} generalizes to $c$ colors yielding
$$\rakc\le c\uparrow^{a-1}(ck-c+1).$$
\end{note}

\section{The \Erdosns-Rado Proof}\label{se:ErdosRado}

Why does Ramsey's proof yield such large upper bounds?
Recall that in Ramsey's proof we do the following:
\begin{itemize}
\item
Color a {\it node} by using Ramsey's theorem (on graphs). This cuts the number of nodes
down by a log (from $m$ to $\Theta(\log m)$).
This is done $2k-1$ times.
\item
After the nodes are colored we use PHP once.
This will cut the number of nodes in half.
\end{itemize}

The key to the large bounds is the number of times we use Ramsey's theorem.
The key insight of the proof by \Erdos and Rado~\cite{ErdosRado} is that
they use PHP many times but Ramsey's theorem only once. In summary they do the following:

\begin{itemize}
\item
Color an {\it edge } by using PHP. This cuts the number of nodes in half.
This is done $\rtwkmt+1$ times.
\item
After {\it all} the edges of a complete graph are colored we use Ramsey's theorem.
This will cut the number of nodes down by a log.
\end{itemize}

We now proceed formally.

\begin{theorem}\label{th:ErdosRado3}~
For almost all $k$, $\rtkt \le \erv$.
\end{theorem}

\begin{proof}

Let $n$ be a number to be determined.
Let $COL$ be a 2-coloring of $\binom{[n]}{3}$.
We define a sequence of vertices, 
$$x_1,x_2,\ldots,x_{\rtwkmt+1}.$$ 

Recall the definition of a 1-homogeneous set for a coloring
of singletons from the note following Definition \ref{de:homog}.
We will use it here.

Here is the intuition: Let $x_1=1$.
Let $x_2=2$.
The vertices $x_1,x_2$
induces the following coloring of $\{3,\ldots,n\}$.
$$COL^*(y) = COL(x_1,x_2,y).$$
Let $V_1$ be a 1-homogeneous for $COL^*$ of size at least $\frac{n-2}{2}$.
Let $COL^{**}(x_1,x_2)$ be the color of $V_1$.
Let $x_3$ be the least vertex left (bigger than $x_2$).

The number $x_3$ induces {\it two} colorings of $V_1 - \{x_3\}$:
$$(\forall y\in V_1-\{x_3\})[COL_1^*(y) = COL(x_1,x_3,y)]$$
$$(\forall y\in V_1-\{x_3\})[COL_2^*(y) = COL(x_2,x_3,y)]$$

Let $V_2$ be a 1-homogeneous for $COL_1^*$ of size $\frac{|V_1|-1}{2}$.
Let $COL^{**}(x_1,x_3)$ be the color of $V_2$.
Restrict $COL_2^*$ to elements of $V_2$, though still call it $COL_2^*$.
We reuse the variable name $V_2$
to be a 1-homogeneous for $COL_2^*$ of size
at least  $\frac{|V_2|}{2}$.
Let $COL^{**}(x_1,x_3)$ be the color of $V_2$.
Let $x_4$ be the least element of $V_2$.
Repeat the process.

We describe the construction formally.

\bigskip

\noindent
{\bf CONSTRUCTION}

\[
\begin{array}{rl}
x_1 =& 1 \cr
V_1= & [n]-\{x_1\}\cr
\end{array}
\]

Let $2\le i\le \rtwkmt+1$. Assume that $x_1,\ldots,x_{i-1},V_{i-1}$, and
$COL^{**}:\binom{\{x_1,\ldots,x_{i-1}\}}{2}\into \{\REDns,\BLUEns\}$ are defined.

\[
\begin{array}{rl}
x_{i} = & \hbox{ the least element of $V_{i-1}$ } \cr
V_{i} =&  V_{i-1} - \{x_{i}\} \hbox{ (We will change this set without changing its name). }\cr
\end{array}
\]

We define $COL^{**}(x_1,x_{i})$, $COL^{**}(x_2,x_{i})$, $\ldots$, $COL^{**}(x_{i-1},x_{i})$.
We will also define smaller and smaller sets $V_i$.
We will keep the variable name $V_i$ throughout.

\bigskip

\noindent
For $j=1$ to $i-1$
\begin{enumerate}
\item
$COL^*: V_{i}\into \{\REDns,\BLUEns\}$ is defined by $COL^*(y) = COL(x_j,x_{i},y)$.
\item
Let $V_{i}$ be redefined as the largest 1-homogeneous set for $COL^*$.
Note that $|V_{i}|$ decreases by at most half.
\item
$COL^{**}(x_j,x_{i})$ is the color of $V_{i}$.
\end{enumerate}

\noindent
KEY: For all $1\le i_1<i_2 \le i$, for all $y\in V_{i}$, 
$COL(x_{i_1},x_{i_2},y)=COL^{**}(x_{i_1},x_{i_2})$.

\noindent
{\bf END OF CONSTRUCTION}

\bigskip

When we derive upper bounds on $n$ 
we will show that the the construction can be carried out for $\rtwkmt+1$ stages. For now assume the construction ends.  

We have vertices

$$X=\{x_1,x_2,\ldots,x_{\rtwkmt+1}\}$$

\noindent
and a 2-coloring $COL^{**}$ of $\binom{X}{2}$.
By the definition of $\rtwkmt+1$ there exists
a set

$$H=\{x_{i_1},\ldots,x_{i_{k}}\}.$$
such that the first $k-1$ elements of it are a 2-homogenous set for $COL^{**}$.
Let the color of this 2-homogenous set be $\REDns$.
We show that $H$ (including $x_{i_k}$) is a 3-homogenous set for $COL$.
For notational convenience we show 
that $COL(x_{i_1},x_{i_2},x_{i_3})=\REDns$.
The proof for any $3$-set of $H$ is similar.

By the definition of $COL^{**}$ 
for all $y\in V_{i_2}$, $COL(x_{i_1},x_{i_2},y)=COL^{**}(x_{i_1},x_{i_2})=\REDns$.
In particular $COL(x_{i_1},x_{i_2},x_{i_3})=\REDns.$

We now see how large $n$ must be so that the construction be carried out.
Note that in stage $i$ $|V_i|$ be decreases by at most half, $i$ times.
Hence $|V_{i+1}|\ge \frac{|V_i|}{2^{i}}$.

Therefore
$$|V_i| \ge \frac{|V_1|}{2^{1+2+\cdots+(i-1)}}\ge \frac{n-1}{2^{(i-1)^2}}.$$

We want $|V_{\rtwkmt+1}|\ge 1$. 
It suffice so take 
$n=2^{{\rtwkmt}^2}+1$.

By Theorem~\ref{th:ramsey2}

$$\rtwkmt^2+1 \le (2^{2k-0.5\lg(k-2)})^2 \le 2^{4k-\lg(k-2)}.$$

Hence
$$
\rtkt\le \erv.
$$
\end{proof}

\begin{note}
A slightly better upper bound for $\rtkt$ can be obtained
by using Conlon's upper bound on $\rtwkt$ given in Note~\ref{no:ramsey2}.
\end{note}

\begin{note}
The proof of Theorem~\ref{th:ErdosRado3} generalizes to $c$-colors yielding 
$$\rtkc\le c^{c^{2ck-\log_c(k-2)+O(c)}}.$$
\end{note}

We state Ramsey's theorem on $a$-hypergraphs~\cite{Ramsey} 
(see also~\cite{GRS,RamseyInts}).

\begin{theorem}\label{th:ErdosRadoa}~
\begin{enumerate}
\item
For all $a\ge 2$, for all $k$, 
$\rakt \le  2^{\binom{\ramkmt+1}{a-1}}+a-2.$
\item
$\rtkt \le \TOW(1,4k-\lg(k-2))$.
\item
For all $a\ge 4$, for almost all $k$, 
$$\rakt \le \TOW(1,a-1,a-2,\ldots,3,4k-\lg (k-a+1)-4(a-3)).$$
\end{enumerate}
\end{theorem}

\begin{proof}

\noindent
1) Assume that $\ramkmt$ exists and $a\ge 2$.

\noindent
{\bf CONSTRUCTION}

\[
\begin{array}{rl}
x_1 =& 1 \cr
\vdots = & \vdots \cr
x_{a-2} =& a-2 \cr
V_{a-2}= & [n]-\{x_1,\ldots,x_{a-2}\}. \hbox{ We start indexing here for convenience.}\cr
\end{array}
\]

Let $a-1\le i\le \ramkmt+1$. Assume that $x_1,\ldots,x_{i-1},V_{i-1}$, and
$COL^{**}:\binom{\{x_1,\ldots,x_{i-1}\}}{a-1}\into \{\REDns,\BLUEns\}$ are defined.

\[
\begin{array}{rl}
x_{i} = & \hbox{ the least element of $V_{i-1}$ } \cr
V_{i} =&  V_{i-1} - \{x_{i}\} \hbox{ (We will change this set without changing its name). }\cr
\end{array}
\]

We define $COL^{**}(A\cup\{x_{i}\})$ for every $A\in \binom{\{x_1,\ldots,x_{i-1}\}}{a-1}$.
We will also define smaller and smaller sets $V_i$.

\noindent
For $A\in \binom{\{x_1,\ldots,x_{i-1}\}}{a-1}$
\begin{enumerate}
\item
$COL^*: V_{i}\into \{\REDns,\BLUEns\}$ is defined by $COL^*(y) = COL(A\cup \{y\})$.
\item
Let $V_{i}$ be redefined as the largest 1-homogeneous set for $COL^*$.
Note that $|V_{i}|$ decreases by at most half.
\item
$COL^{**}(A\cup \{x_{i}\})$ is the color of $V_{i}$.
\end{enumerate}

\noindent
KEY: For all $l\le i_1 < \cdots < i_a\le i$,
$COL(x_{i_1},\ldots,x_{i_a}) = COL^{**}(x_{i_1},\ldots,x_{i_{a-1}})$.

\noindent
{\bf END OF CONSTRUCTION}

When we derive upper bounds on $n$ 
we will show that the the construction can be carried out for $\ramkmt+1$ stages. For now assume the construction ends.  

We have vertices

$$X=\{x_1,x_2,\ldots,x_{\ramkmt+1}\}$$

\noindent
and a 2-coloring $COL^{**}$ of $\binom{X}{2}$.
By the definition of $\ramkmt+1$ there exists
a set

$$H=\{x_{i_1},\ldots,x_{i_{k}}\}.$$
such that the first $k-1$ elements of it are a $(a-1)$-homogenous set for $COL^{**}$.
Let the color of this $(a-1)$-homogenous set be $\REDns$.
We show that $H$ (including $x_{i_k}$) is a $a$-homogenous set for $COL$.
For notational convenience we show 
that $COL(x_{i_1},\ldots,x_{i_a})=\REDns$.
The proof for any $a$-set of $H$ is similar, including the case where the last vertex is $x_{i_k}$.

By the definition of $COL^{**}$ 
for all $y\in V_{i_2}$, $COL(x_{i_1},\ldots,x_{i_{a-1}},y)=COL^{**}(x_{i_1},\ldots,x_{i_{a-1}})=\REDns$.
In particular $COL(x_{i_1},\ldots,x_{i_a})=\REDns.$

We now see how large $n$ must be so that the construction can be carried out.
Note that 
during stage $i$ there will be 
$\binom{i}{a-2}$  times where 
$|V_i|$ decreases by at most half.
Hence $|V_{i+1}|\ge \frac{|V_i|}{2^{\binom{i}{a-2}}}$.

Therefore
$$
|V_i|\ge\frac{|V_{a-2}|}{2^{\binom{a-2}{a-2}+\binom{a-1}{a-2}+\binom{a}{a-2}+\cdots+\binom{i-1}{a-2}}}=\frac{n-a+2}{2^{\binom{i}{a-1}}}.
$$

We want $|V_{\ramkmt+1}|\ge 1$. 

Hence we need

$$
|V_{\ramkmt+1}|\ge\frac{n-a+2}{2^{\binom{\ramkmt+1}{a-1}}}\ge 1
$$

$$
n-a+2 \ge 2^{\binom{\ramkmt+1}{a-1}}
$$

Hence

$$
n\ge 2^{\binom{\ramkmt+1}{a-1}}+a-2.
$$

Therefore
$$\rakt \le  2^{\binom{\ramkmt+1}{a-1}}+a-2.$$

\noindent
2) This is a restatement of Theorem~\ref{th:ErdosRado3}.

\noindent
3) We use Lemma~\ref{le:tow} throughout this proof implicitly.
We will also use a weak form of the recurrence from Part 1, namely:
$$\rakt \le 2^{\ramkmt^{a-1}}.$$

We prove the bound on $\rakt$ for $a\ge 4$ by induction on $a$.

\noindent
{\bf Base Case: $a=4$:} 
By Part 2, $\rtkt \le \TOW(1,4k-\lg(k-2))$. Hence

$$\rfokt \le 2^{\rtkmt^2} \le \TOW(1,3,4k-\log(k-3)-4)=\TOW(1,3,4k-\log(k-3)-4\times(4-3)).$$

\noindent
{\bf Induction Step:} We assume 

\[
\begin{array}{rl}
\ramkmt \le & \TOW(1,a-2,\ldots,3,4(k-1)-\lg ((k-1)-(a-1)+1)-4(a-4))\cr
         =  & \TOW(1,a-2,\ldots,3,4k-4-\lg (k-a+1)-4(a-3)).\cr
\end{array}
\]

Hence

$$\rakt \le 2^{\ramkmt^{a-1}} \le  \TOW(1,a-1,a-2,\ldots,3,4k-\lg (k-a+1)-4(a-3)).$$
\end{proof}

\begin{corollary}
For all $a\ge 3$, for almost all $k$,
$\rakt\le \TOW(1,1,\ldots,1,4k)$ where there are $a-2$ 1's.
(This is often called {\it $2$ to the $2$ to the 2 $\ldots$, $a-2$ times
and then a $4k$ at the top.})
\end{corollary}

\begin{note}
The proof of Theorem~\ref{th:ErdosRadoa} easily generalizes to yield the following.
\begin{enumerate}
\item
For all $a\ge 2$, for all $k$, 
$\rakc \le  c^{\binom{\ramkmc+1}{a-1}}+a-2.$
\item
$\rtkc\le \TOW_c(1,2ck-\log_c(k-2)+O(c))$.
\item
For all $a\ge 4$, for almost all $k$, 
$$\rakc \le \TOW_{c}(1,a-1,a-2,\ldots,3,2ck-\log_c(k-a+1)+O(c)).$$
\end{enumerate}
\end{note}

\section{The Conlon-Fox-Sudakov Proof}\label{se:cfs}
Recall the following high level description of the \Erdosns-Rado proof:

\begin{itemize}
\item
Color an edge by using PHP. This cuts the number of nodes in half.
This is done $\rtwkmt+1$ times.
\item
After {\it all} the edges of a complete graph are colored we use Ramsey's theorem.
This will cut the number of nodes down by a log.
\end{itemize}

{\it Every} time we colored an edge we cut the number of vertices in half.
Could we color fewer edges? Consider the following scenario:

$COL^{**}(x_1,x_2)=\REDns$ and $COL^{**}(x_1,x_3)=\BLUEns$. Intuitively the edge from $x_2$ to $x_3$
might not be that useful to us.
{\it Therefore we will not color that edge!}

Two questions come to mind:

\bigskip

\noindent
{\bf Question:} How will we determine which edges are potentially useful?

\noindent
{\bf Answer:} We will associate to each $x_i$ a 2-colored 1-hypergraph $G_i$
that keeps track of which edges $(x_{i'},x_i)$ are colored, and if so what
they are colored. For example, if $COL^{**}(x_7,x_9)=\REDns$ then
$(7,\REDns)\in G_9$.
(We use the terminology {\it 2-colored 1-hypergraphs} and the notation $G_i$
so that when we extend this to the 
$a$-hypergraph Ramsey Theorem, in the appendix,
the similarity will be clear.)

We will have $x_1=1$ and $G_1=\es$.
Say we already have

$$x_1,\ldots,x_i$$

$$G_1,\ldots,G_i.$$

\noindent
Assume $i'<i$.
Assume that for each of 
$COL^{**}(x_1,x_i),\ldots,COL^{**}(x_{i'-1},x_i)$
we have either defined it or intentionally chose to not define it.
We are wondering if we should define $COL^{**}(x_{i'},x_i)$.
At this point 
the vertices of $G_i$ are a subsets of $\{1,\ldots,i'-1\}$.
If $G_i$ is equal (not just isomorphic) to $G_{i'}$ (as colored 1-hypergraphs)
then we will define $COL^{**}(x_{i'},x_i)$ and add 
$i'$ to $G_i$ with that color.
If $G_i$ is not equal to $G_{i'}$ 
then we will not define $COL^{**}(x_{i'},x_i)$.

\bigskip

\noindent
{\bf Question:} Since we only color some of the edges how will we use Ramsey's theorem?

\noindent
{\bf Answer:} We will not. Instead we go until one of the 1-hypergraphs has 
$k$ monochromatic points. Hence we will be using the 1-ary Ramsey Theorem.
(When we prove the $a$-hypergraph Ramsey theorem we will use the $(a-2)$-hypergraph Ramsey Theorem.)

We need a lemma that will help us in both the case of $c=2$ and
the case of general $c$.

\begin{lemma}\label{le:sigmac}
Let $S\subseteq [c]^\kstar$ be such that no string in $S$ has
$\ge k-1$ of any $i\in [c]$.
Then the following hold:
\begin{enumerate}
\item
$$
\sum_{\sigma\in S} |\sigma| \le \displaystyle k^{3/2-c/2}c^{c(k-1) + 2} \left(\frac{e}{\sqrt{2 \pi}} \right)^{c+1}$$
\item
If $c=2$ then the summation is bounded above by
$\bigl(\frac{e}{\sqrt{2 \pi}}\bigr)^3k^{1/2}2^{2k}$.
\end{enumerate}
\end{lemma}

\begin{proof}

Let
$$A = \sum \{ |\sigma| \ :\ \sigma \in [c]^*, \sigma \hbox{ contains at most } k-1 \hbox{ of any element} \}.$$

Grouping by the number of appearances of each element of $[c]$, we get
\[
A = \sum_{j_1 = 0}^{k-1} \cdots \sum_{j_c = 0}^{k-1} (j_1 + \ldots + j_c)
    \frac{(j_1 + \ldots + j_c)!}{j_1! \cdots j_c!}.
\]

We may split up the innermost sum to get $c$ different sums, each containing
a single $j_i$ in the summand. Since each of these sums is equal, we get

\[
A = \displaystyle
    c \cdot \sum_{j_1 = 0}^{k-1} j_1 \cdot \sum_{j_2 = 0}^{k-1} \cdots
    \sum_{j_c = 0}^{k-1} \frac{(j_1 + \ldots + j_c)!}{j_1! \cdots j_c!}.
\]

We split this up into the part which depends on $j_c$, and the part which
doesn't:
\begin{equation}\label{eq:basecase}
A = \displaystyle
    c \cdot \sum_{j_1 = 0}^{k-1} j_1 \cdot \sum_{j_2 = 0}^{k-1} \cdots
    \sum_{j_c = 0}^{k-1} \frac{(j_1 + \ldots + j_{c-1})!}{j_1! \cdots j_{c-1}!} \binom{j_1 + \ldots + j_c}{j_c}.
\end{equation}

	\paragraph{Claim} For all $\ell$, with $0 \le \ell \le c-1$,
	\[
	A \le \displaystyle c \cdot \sum_{j_1 = 0}^{k-1}
	      \frac{j_1}{\prod_{i=0}^{\ell-1} (j_1+ik)} \cdot \sum_{j_2 = 0}^{k-1}
	      \cdots \sum_{j_{c - \ell} = 0}^{k-1} B_\ell \cdot
	      \binom{j_1 + \ldots + j_{c-\ell} + \ell k}{j_{c-\ell}},
	\]
where
\[
B_\ell = \frac{(j_1 + \ldots + j_{c-\ell-1} + \ell k)!}
         {j_1! \cdots j_{c-\ell-1}! (k-1)!^{\ell}}
\]
does not depend on $j_{c - \ell}$. Note that, in the case $\ell = c-1$,
all the inner sums are gone, so we are left with
\[
A \le \displaystyle c \cdot \sum_{j_1 = 0}^{k-1}
      \frac{j_1}{\prod_{i=0}^{c-2} (j_1+ik)} \cdot B_{c-1} \cdot
      \binom{j_1 + (c-1) k}{j_1}.
\]

\noindent
{\bf Proof of Claim}

We will prove this by induction on $\ell$. The base case is
Equation~\ref{eq:basecase}.

For the inductive step, we need only to look at the innermost sum,
whose value we call $S$.

\[
\begin{array}{rcl}
S & = & \displaystyle \sum_{j_{c-\ell}=0}^{k-1} B_\ell \cdot
        \binom{j_1 + \ldots + j_{c-\ell} + \ell k}{j_{c-\ell}} \\
\noalign{\bigskip}
  & = & \displaystyle B_\ell \cdot \sum_{j_{c-\ell}=0}^{k-1}
        \binom{j_1 + \ldots + j_{c-\ell} + \ell k}{j_{c-\ell}} \\
\noalign{\bigskip}
  & = & \displaystyle
        B_\ell \cdot \binom{j_1 + \ldots + j_{c-\ell-1} + \ell k + k}{k-1}. \\
\end{array}
\]

Here we used Pascal's $2^{\hbox{nd}}$ Identity:
\[
\sum_{b=0}^n \binom{a+b}{b} = \binom{a+n+1}{n}.
\]
with $a = j_1 + \ldots + j_{c-\ell-1} + \ell k$.

Writing our answer out in terms of factorials, we get
\[
\begin{array}{rcl}
S & = & \displaystyle \frac{(j_1 + \ldots + j_{c-\ell-1} + \ell k)!}
        {j_1! \cdots j_{c-\ell-1}! (k-1)!^{\ell}} \cdot
        \frac{(j_1 + \ldots + j_{c-\ell-1} + (\ell+1) k)!}
        {(k-1)!(j_1 + \ldots + j_{c-\ell-1} + \ell k + 1)!} \\
\noalign{\bigskip}
  & = & \displaystyle \frac{(j_1 + \ldots + j_{c-\ell-1} + \ell k)!}
        {(j_1 + \ldots + j_{c-\ell-1} + \ell k + 1)!} \cdot
        \frac{(j_1 + \ldots + j_{c-\ell-1} + (\ell+1) k)!}
        {j_1! \cdots j_{c-\ell-1}! (k-1)!^{\ell+1}} \\
\noalign{\bigskip}
  & = & \displaystyle
        \left( \frac{1}{j_1 + \ldots + j_{c-\ell-1} + \ell k + 1} \right)
        \cdot \frac{(j_1 + \ldots + j_{c-\ell-1} + (\ell+1) k)!}
        {j_1! \cdots j_{c-\ell-1}! (k-1)!^{\ell+1}} \\
\noalign{\bigskip}
  & \le & \displaystyle
        \left( \frac{1}{j_1 + \ell k} \right) \cdot
        \frac{(j_1 + \ldots + j_{c-\ell-1} + (\ell+1) k)!}
        {j_1! \cdots j_{c-\ell-1}! (k-1)!^{\ell+1}} \\
\noalign{\bigskip}
  & = & \displaystyle
        \left( \frac{1}{j_1 + \ell k} \right)
        \left( \frac{(j_1 + \ldots + j_{c-\ell-2} + (\ell+1) k)!}
        {j_1! \cdots j_{c-\ell-2}! (k-1)!^{\ell+1}} \right)
        \binom{j_1 + \ldots + j_{c-\ell-1} + (\ell+1) k}{j_{c - \ell - 1}}. \\
\noalign{\bigskip}
  & = & \displaystyle
        \left( \frac{1}{j_1 + \ell k} \right)
        B_{\ell + 1}
        \binom{j_1 + \ldots + j_{c-\ell-1} + (\ell+1) k}{j_{c - \ell - 1}}. \\
\end{array}
\]

Reinserting this value $S$ back into the formula for $A$, and factoring
the fraction $\frac{1}{j_1 + \ell k}$ to the outermost sum, we get the
desired result.

The induction stops when we hit the outermost sum, where the format of the
summand changes.

\noindent
{\bf End of Proof of Claim}

Using this claim, with $\ell = c-1$, we get the bound
\[
A \le \displaystyle c \cdot \sum_{j_1 = 0}^{k-1}
      \frac{j_1}{\prod_{i=0}^{c-2} (j_1+ik)} \cdot B_{c-1} \cdot
      \binom{j_1 + (c-1) k}{j_1}.
\]
Note the first fraction: the $j_1$ in the numerator cancels with the $i=0$
term of the denominator. As for the rest of the terms, they reach their maxima
when $j_1=0$.

Renaming $j_1$ to be $n$ and filling in the value of $B_{c-1}$, we get
\allowdisplaybreaks
\[
\begin{array}{rcl}
A & \le & \displaystyle c \cdot \sum_{n = 0}^{k-1}
          \frac{n}{\prod_{i=0}^{c-2} (n+ik)}
          \cdot \frac{((c-1)k)!}{(k-1)!^{c-1}} \cdot \binom{n + (c-1) k}{n} \\
\noalign{\bigskip}
  & \le & \displaystyle \frac{c}{\prod_{i=1}^{c-2} ik} \cdot
          \frac{((c-1)k)!}{(k-1)!^{c-1}}
          \sum_{n = 0}^{k-1} \binom{n + (c-1) k}{n} \\
\noalign{\bigskip}
  &  =  & \displaystyle \frac{c}{(c-2)! k^{c-2}} \cdot
          \frac{((c-1)k)!}{(k-1)!^{c-1}}
          \sum_{n = 0}^{k-1} \binom{n + (c-1) k}{n} \\
\noalign{\bigskip}
  &  =  & \displaystyle \frac{c}{(c-2)! k^{c-2}} \cdot
          \frac{((c-1)k)!}{(k-1)!^{c-1}} \cdot \binom{ck}{k-1} \\
\noalign{\bigskip}
  &  =  & \displaystyle \frac{c}{(c-2)! k^{c-2}} \cdot
          \frac{((c-1)k)!}{(k-1)!^{c-1}} \cdot
          \frac{(ck)!}{(k-1)!((c-1)k + 1)!} \\
\noalign{\bigskip}
  & \le & \displaystyle \frac{c}{(c-2)! k^{c-2}} \cdot
          \frac{((c-1)k)!}{((c-1)k + 1)!} \cdot
          \frac{(ck)!}{(k-1)!^c} \\
\noalign{\bigskip}
  &  =  & \displaystyle \frac{c}{(c-2)! k^{c-2}} \cdot
          \frac{1}{(c-1)k + 1} \cdot
          \frac{(ck)!}{(k-1)!^c} \\
\noalign{\bigskip}
  & \le & \displaystyle \frac{c^2 k}{c!} \cdot
          \frac{(ck)!}{k!^c} \\
\end{array}
\]

Now we use the bounds associated with Stirling's approximation:
\[
\sqrt{2 \pi n} \left(\frac{n}{e}\right)^n \le n! \le
e \sqrt{n} \left(\frac{n}{e}\right)^n
\]

\[
\begin{array}{rcl}
A & \le & \displaystyle \frac{c^2 k}{c!} \cdot
          \frac{e (ck)^{1/2} (ck)^{ck}}{(2 \pi k)^{c/2} k^{ck}} \\
\noalign{\bigskip}
  & \le & \displaystyle
          \frac{e}{c!} \cdot c^{ck + 5/2} k^{3/2 - c/2} (2 \pi)^{-c/2} \\
\noalign{\bigskip}
  & \le & \displaystyle
          \frac{e}{\sqrt{2 \pi c} (c/e)^c} \cdot
          c^{ck + 5/2} k^{3/2 - c/2} (2 \pi)^{-c/2} \\
\noalign{\bigskip}
  & \le & \displaystyle c^{c(k-1) + 2} k^{3/2 - c/2}
          \left(\frac{e}{\sqrt{2 \pi}} \right)^{c+1}. \\
\end{array}
\]

\end{proof}

The following proof is by Conlon-Fox-Sudakov~\cite{conlonfoxsud}; however,
we do a more careful analysis with the aide of Lemma~\ref{le:sigmac}.2.

\begin{theorem}\label{th:cfs3}~
For all $k$, $\rtkt \le \cfsv$ where
$B=(\frac{e}{\sqrt{2\pi}})^3 \sim 1.28$.
\end{theorem}

\begin{proof}
Let $n$  be a number to be determined.
Let $COL$ be a 2-coloring of $\binom{[n]}{3}$.

We define a finite sequence of vertices $x_1,x_2,\ldots, x_L$ where
we will bound $L$ later.
For every $1\le i\le L$ we will also define $G_i$, 
a 2-colored 1-hypergraph. We will represent $G_i$ as a subset of $\nat\times \{\REDns,\BLUEns\}$.
For example, $G_i$ could be $\{ (1,\REDns), (4,\BLUEns), (5,\REDns) \}$.
The notation $G_i = G_i \cup \{ (12,\REDns) \}$ means that we 
add the edge $\{12\}$ to $G_i$ and color
it $\REDns$. When we refer to the vertices of the $G_i$ 1-hypergraph we will often refer
to them as {\it 1-edges} since (1) in a 1-hypergraph,
vertices are edges, and (2) the proof will generalize to $a$-hypergraphs more
easily. We use the term 1-edges so the reader will remember they are vertices also.

The construction will stop when one of the $G_i$ has 
a 1-homogenous set of size $k-1$ (more commonly called a set of
$k-1$ monochromatic points).
We will later show that this must happen.

Recall the definition of a 1-homogeneous set relative to a coloring
of a 1-hypergraph from the note following Definition \ref{de:homog}.
We will use it here.

Here is the intuition: 
Let $x_1=1$ and $x_2=2$.
Let $G_1=\es$.
The vertices $x_1,x_2$
induces the following coloring of $\{3,\ldots,n\}$.
$$COL^*(y) = COL(x_1,x_2,y).$$
Let $V_1$ be a 1-homogeneous set of size at least $\frac{n-2}{2}$.
We will only work within $V_1$ from now on.
Let $COL^{**}(x_1,x_2)$ be the color of $V_1$.
Let $G_2=\{(1,COL^{**}(x_1,x_2)\}$.

Let $x_3$ be the least vertex in $V_1$.
The number $x_3$ induces {\it two} colorings of $V_1 - \{x_3\}$:
$$COL_{1,3}^*(y) = COL(x_1,x_3,y)$$
$$COL_{2,3}^*(y) = COL(x_2,x_3,y)$$

Let $V_2$ be a 1-homogeneous for $COL_{1,3}^*$ of size $\frac{|V_1|-1}{2}$.
Let $COL^{**}(x_1,x_3)$ be the color of $V_2$.
We also set 
$G_3=\{(1,COL^{**}(x_1,x_3))\}$, though we will may add to $G_3$ later.
Restrict $COL_{2,3}^*$ to elements of $V_2$, though still call it $COL_{2,3}^*$.
We will only work within $V_2$ from now on.

Will we color $(x_2,x_3)$? If $G_2=G_3$ (that is, if they both colored 1 the same) then YES.
If not then we won't. This is the KEY--- every time we color an edge we
divide $V$ in half. We will not always color an edge- only the promising ones.
Hence $V$ will  not decrease as quickly as was done in the proof of Theorem~\ref{th:ErdosRado3}.

If $G_2=G_3$ then we reuse the variable name $V_2$ to be 
a 1-homogeneous for $COL_{2,3}^*$ of size
at least  $\frac{|V_2|}{2}$.
Let $COL^{**}(x_2,x_3)$ be the color of $V_2$.
Add $(2,COL^{**}(x_2,x_3))$ to $G_3$.

If $G_2\ne G_3$ then we do not color $(x_2,x_3)$ and do not add anything to $G_3$.

In the actual construction we will not define $COL^{**}$ since the information it contains
will be stored in the 2-colored 1-hypergraphs $G_i$.

We describe the construction formally.

\begin{definition}\label{de:agree1}
Let $G_{i_1},G_{i_2}$ be 2-colored 1-hypergraphs. 
Let $j\in\nat$.
\begin{enumerate}
\item
$G_{i_1}$ and $G_{i_2}$ {\it agree on $j$} if,
either 
(1) $G_{i_1}$ and $G_{i_2}$ both have 1-edge $j$ and color it the same or,
(2)  neither $G_{i_1}$ nor $G_{i_2}$ has 1-edge $j$.
\item
$G_{i_1}$ and $G_{i_2}$ {\it agree on $\{1,\ldots,j\}$} if $G_{i_1}$ and $G_{i_2}$
agree on all of the 1-edges in the set $\{1,\ldots,j\}$.
\item
$G_{i_1}$ and $G_{i_2}$ {\it disagree on $j$} if
either 
(1) $G_{i_1}$ and $G_{i_2}$ both have 1-edge $j$ and color it differently or 
(2) one of them has 1-edge $j$ but the other one does not.
\end{enumerate}
\end{definition}

\noindent
{\bf CONSTRUCTION}

\[
\begin{array}{rl}
x_1 =& 1 \cr
x_2 =& 2   \cr
G_1= & \es \cr
V_1= &[n] - \{x_1,x_2\}\cr
COL^*(y)= & COL(x_1,x_2,y) \hbox{ for all } y\in V_1 \cr
V_2 = & \hbox{ the largest 1-homogeneous set for $COL^*$ }\cr
G_2 = & \{ (1,\hbox{ the color of $V_2$}) \}\cr
\end{array}
\]

\noindent
KEY: for all $y\in V_2$, $COL(x_1,x_2,y)$ is the color of $1$ in $G_2$.

Let $i\ge 2$, and assume that 
$V_{i-1}$, $x_1,\ldots,x_{i-1}$,
$G_1,\ldots,G_{i-1}$ are defined.
If $G_{i-1}$ has a 1-homogenous set of size $k-1$ then stop (yes, $k-1$- this is not a typo).
Otherwise proceed.

\[
\begin{array}{rl}
G_i = & \es \hbox{ (This will change.) } \cr
x_{i} = & \hbox{ the least element of $V_{i-1}$ } \cr
V_{i} =&  V_{i-1} - \{x_{i}\} \hbox{ (We will change this set without changing its name.) }\cr
\end{array}
\]

We will add some colored 1-edges to $G_i$.
We will also define smaller and smaller sets $V_i$.
We will keep the variable name $V_{i}$ throughout.

\bigskip

\noindent
For $j=0$ to $i-1$
\begin{enumerate}
\item
If $G_j=G_i$ then proceed, else go to the next value of $j$.
(Note that we are asking if $G_j=G_i$ at a time when $G_i$'s vertex set
is a subset of $\{1,\ldots,j-1\}$.)
\item
$COL^*: V_{i}\into \{\REDns,\BLUEns\}$ is defined by $COL^*(y) = COL(x_j,x_{i},y)$.
\item
$V_{i}$ is the largest 1-homogeneous set for $COL^*$.
Note that $|V_{i}|$ decreases by at most half.
\item
$G_i = G_i \union \{ (j, \hbox{ color of $V_i$} ) \}$ 
\end{enumerate}

\noindent
KEY: 
Let $1\le i_1<i_2 \le i$ such that $i_1$ is a 1-edge of $G_{i_2}$. Let $c_{i_1}$ 
be such that $(i_1,c_{i_1})\in G_{i_2}$.
For all $y\in V_{i}$, $COL(x_{i_1},x_{i_2},y)=c_{i_1}$. 

\noindent
{\bf END OF CONSTRUCTION}

When we derive upper bounds on $n$ we will show that the construction ends.  For now assume the construction ends.

When the construction ends we have a $G_L$ that
has a 1-homogenous set of size $k-1$.
We assume the color is $\REDns$.
Let $\{i_1 < i_2 <\cdots < i_{k-1} \}$ be the 1-homogenous set.
Define $i_k=L$.
We show that 
$$H= \{ x_{i_1},\ldots,x_{i_k} \}$$
is a 3-homogenous set with respect to the original coloring $COL$.
For notational convenience we show 
that $COL(x_{i_1},x_{i_2},x_{i_3})=\REDns$.
The proof for any 3-set of $H$ is similar, 
even for the case where the last point is $x_L$.

Look at $G_{i_2}$. Since $i_2$ is a 1-edge in $G_L$ 
we know that
$G_{i_2}$ and $G_L$ agree on all 1-edges in $\{1,\ldots,i_2-1\}$.
Since $(i_1,\REDns)\in G_L$ and $i_1\le i_2-1$, $(i_1,\REDns)\in G_{i_2}$.
Hence, for all $y\in V_{i_2}$, $COL(x_{i_1},x_{i_2},y)=\REDns$.
In particular $COL(x_{i_1},x_{i_2},x_{i_3})=\REDns$.

We now establish bounds on $n$.

\begin{definition}
Let $G=V$ be a 2-colored 1-hypergraph on vertex set $V=\{L_1<\cdots<L_m\}$ and edge set $E$.
Define $\squash(G)$ to be $G'=(V',E')$, the following 2-colored 1-hypergraph:
\begin{itemize}
\item
The vertex sets $V'=\{1,\ldots,m\}$.
\item
For each edge $\{L_{i}\}$  in $E$
the edge $\{i\}$ is in $E'$.
\item
The color of $\{i\}$ in $G'$ is the color of 
$\{L_i\}$ in $G$.
\end{itemize}
\end{definition}

\noindent
{\bf Claim 1:} For all $2\le i_1<i_2$, $\squash(G_{i_1})\ne\squash(G_{i_2})$.

\noindent
{\bf Proof of Claim 1:}
Assume, by way of contradiction, that $i_1<i_2$ and $\squash(G_{i_1})=\squash(G_{i_2})$.
Let $G_{i_1}$  have vertex set $U_1$.
Let $f_1$ be the isomorphism that maps $U_1$ to the vertex set of $\squash(G_{i_1}$).
Note that $f_1$  is order preserving.
If $f_1$ is applied to a number not in $U_1$ then the result is undefined.
Let $U_2$ and $f_2$ be defined similarly for $G_{i_2}$.

We will prove that, for all $1\le j\le i_1-1$, 
(1) $f_1$ and $f_2$ agree on $\{1,\ldots,j\}$,
(2) $G_{i_1}$ and $G_{i_2}$ agree on $\{1,\ldots,j\}$.
The proof will be by induction on $j$.

\noindent
{\bf Base Case: $j=1$}. 
Since $2\le i_1,i_2$, the edge $E=\{1\}$ is in both $G_{i_1}$ and $G_{i_2}$,
hence $f_1(1)=f_2(1)$.
If the color of $E$ is different in $G_{i_1}$ and $G_{i_2}$ then 
$\squash(G_{i_1})\ne \squash(G_{i_2})$.
Hence the color of $E$ is the same in both graphs.
Hence $G_{i_1}$ and $G_{i_2}$ agree on $\{1\}$.

\noindent
{\bf Induction Step:} Assume that $G_{i_1}$ and $G_{i_2}$ agree on $\{1,2,\ldots,j-1\}.$
Assume that $f_1$ and $f_2$ agree on $\{1,\ldots,j-1\}$.
We use these assumptions without stating them.
Look at what happens when $G_{i_1}$ ($G_{i_2}$) has to decide what to do with $j$.

If $G_j$ and $G_{i_1}$ agree on \hbox{$\{1,\ldots,j-1\}$} then, since $j<i_1$,
$G_j$ also agrees with $G_{i_2}$ on $\{1,\ldots,j-1\}$.
Hence edge $E=\{j\}$ will be put into both
$G_{i_1}$ and $G_{i_2}$. Hence $j$ will be a vertex in both $G_{i_1}$ and $G_{i_2}$
so $f_1(j)=f_2(j)$. 
Since $f_1$ and $f_2$ agree on $\{1,\ldots,j\}$
and $\squash(G_{i_1})=\squash(G_{i_2})$, $E$ must be the same
color in $G_{i_1}$ and $G_{i_2}$. 
Hence $G_{i_1}$ and $G_{i_2}$ agree on $\{1,\ldots,j\}$.

If $G_j$ does not agree with $G_{i_1}$ on $\{1,\ldots,j-1\}$ 
then there must be an edge $E\in \{1,\ldots,j-1\}$
such that $G_j$ and $G_{i_1}$ disagree on $E$.
Hence $G_j$ and $G_{i_2}$ disagree on $E$.
Thus $j$ will not be made a vertex of $G_{i_1}$ or $G_{i_2}$ ever.
Hence both $f_1(j)$ and $f_2(j)$ are undefined.
The edge $E$ is not added to $G_{i_1}$ or $G_{i_2}$ in stage $j$.
Since $G_{i_1}$ and $G_{i_2}$ agree on $\{1,\ldots,j-1\}$
they agree on $\{1,\ldots,j\}$.

We now know that $G_{i_1}$ and $G_{i_2}$ agree on $\{1,\ldots,i_1-1\}$.
Note that $G_{i_1}$ only has vertices in $\{1,\ldots,i_1-1\}$.
Look at stage $i_1$ in the construction of $G_{i_2}$.
Since $G_{i_1}$ agrees with $G_{i_2}$ on $\{1,\ldots,i_1-1\}$
$i_1$ is an vertex in $G_{i_2}$.
At that point $G_{i_2}$ will have more vertices then $G_{i_1}$ hence
$\squash(G_{i_1})\ne\squash(G_{i_2})$. This is a contradiction.

\noindent
{\bf End of Proof of Claim 1}

We now bound $L$, the length of the sequence.
The sequence 
$G_1,G_2,\ldots,$ will end when some $G_i$ has 
$2k-3$ points in it (so at least $k-1$ must be the same color)
or earlier. 
For all $i$, map $G_i$ to $\squash(G_i)$. This mapping is 1-1 by Claim~1.
Hence the length of the sequence is bounded by the
number of 2-colored 1-hypergraphs on {\it an initial segment of}
$\{1,\ldots,2k-3\}$ so $L\le 2^0 + \cdots + 2^{2k-3} \le 2^{2k-2}-1$.
We have shown the construction terminates.

Strangely enough, this is not quite what we care about when we are bounding $n$.
We care about the number of {\it edges} in all of the $G_i$'s since each
edge at most halves the number of vertices.

By Lemma~\ref{le:sigmac}, the number of edges in all of the $G_i$ is
bounded by $B(k-1)^{1/2}2^{2k}$
where $B=\bigl(\frac{e}{\sqrt{2\pi}}\bigr)^3$.
Hence the number of times $|V|$ is cut in at most half is bounded by
that same quantity.
Hence it suffices to take $n = 2^{B(k-1)^{1/2}2^{2k}}$.

\end{proof}

\begin{note}
For $c\ge 2$ let $B_c= \bigl(\frac{e}{\sqrt{2\pi}}\bigr)^{c+1}$.
The proof of Theorem~\ref{th:cfs3} generalize to $c$ colors yielding
$\rtkc \le c^{B_c (k-1)^{1/2}c^{ck}}$.
\end{note}

\begin{theorem}\label{th:cfsa}~
Throughout this theorem $B=(\frac{e}{\sqrt{2\pi}})^3 \sim 1.28$.
\begin{enumerate}
\item
$\rtkt \le \TOW(B(k-1)^{1/2},2^{2k})$.
\item
$\rfokt \le \TOW(1,3B(k-2)^{1/2},2^{2k-2})$.
\item
$\rfikt \le \TOW(1,4,3B(k-3)^{1/2},2^{2k-4})$.
\item
For all $a\ge 6$, for almost all $k$, 
$$\rakt \le \TOW(1,a-1,a-2,\ldots,4,3B(k-a+2)^{1/2},2^{2k-2a+6})$$
\end{enumerate}
\end{theorem}

\begin{proof}

Part 1 is a restatement of Theorem~\ref{th:cfs3}.

From Theorem~\ref{th:ErdosRadoa} we have $\rakt \le 2^{\ramkmt^{a-1}}.$
We apply this recurrence to Part 1 to get Part 2, and to Part 2 to get Part 3.
We then use it to get Part 4 by induction. 

\end{proof}

\begin{note}
For $c\ge 2$ let $B_c= \bigl(\frac{e}{\sqrt{2\pi}}\bigr)^{c+1}$.
The proof of Theorem~\ref{th:cfsa} generalize to $c$ colors yielding the following.
\begin{enumerate}
\item
$\rtkc \le \TOW_c(B_c(k-1)^{1/2},c^{ck})$.
\item
$\rfoktc \le \TOW_c(1,3B(k-2)^{1/2},c^{ck-c})$.
\item
$\rfiktc \le \TOW_c(1,4,3B(k-3)^{1/2},c^{ck-2c})$.
\item
For all $a\ge 6$, for almost all $k$, 
$$\rakc \le \TOW_c(1,a-1,a-2,\ldots,4,3B(k-a+2)^{1/2},c^{ck-ac+3c})$$
\end{enumerate}
\end{note}

\section{Open Problems}

The best known lower bounds are attributed to \Erdos and Hajnal in~\cite{GRS}.
They are as follows:
\begin{enumerate}
\item
$\rtkt \ge 2^{\Omega(k^2)}$ by a simple probabilistic argument.
\item
$\rakt \ge \TOW(1,\ldots,1,\Omega(k^2))$ ($a-1$ 1's) by the lower bound on $\rtkt$ and
the stepping up lemma.
\end{enumerate}

For 4 colors the situation is very different. \Erdos and Hajnal showed that
$$\rtkf \ge 2^{2^{\Omega(k)}}.$$

Obtaining matching upper and lower bounds for the hypergraph Ramsey
Numbers seems to be a hard open problem.
We suspect that a bound of the form $\rakt \le 2^{2^{k+o(k)}}$ can be obtained.

\section{Acknowledgments}

We would like to thank David Conlon whose talk on this topic at RATLOCC 2011 inspired
this paper. We would also like to thank David Conlon (again), Jacob Fox and Benny Sudakov for their paper~\cite{conlonfoxsud}
which contains the new proof of the 3-hypergraph Ramsey Theorem.
We also thank Jessica Shi and Sam Zbarsky who helped us clarify some of the results.

\appendix
\section{Extending Conlon-Fox-Sudakov to $a$-Hypergraph Ramsey}

In this appendix we extend the Conlon-Fox-Sudakov proof to
prove the $a$-hypergraph Ramsey Theorem. Unfortunately it does
not yield better bounds on $\raktreal$.
We include it in the hope that in the future someone may modify the construction,
or our analysis of it, to yield better bounds.

In order to prove an upper bound on $\rakt$) (and $\rakc$)
we need a lemma similar to Lemma~\ref{le:sigmac}.
The lemma below gives a crude estimate. It is possible that a more careful bound
would lead to a better analysis of the construction and hence to a better bound
on the hypergraph Ramsey numbers.

\begin{lemma}\label{le:hypera}
Let $S$ be the subset of $c$-colored complete $(a-2)$-hypergraphs 
whose vertex sets are an initial
segments of $\nat$ and that have no ($a-2$)-homogenous set
of size $k-1$. Then
$$\sum_{(V,E,COL)\in S} |E| \le {\rammkmc}^{a-1}c^{{\rammkmc}^{a-2}}.$$
\end{lemma}

\begin{proof}

The largest size of $V$ such that a $c$-colored $(a-2)$-hypergraph $(V,E)$ has no ($a-2$)-homogenous
set of size $k-1$ is bounded above by $\rammkmc$. Hence we want to bound.

$$
\sum_{i=1}^{\rammkmc}\sum_{(V,E,COL)\in S,|V|=i}|E|\le\sum_{i=1}^{\rammkmc}\sum_{(V,E,COL)\in S,|V|=i}i^{a-2}.
$$

The number of $c$-colored $(a-2)$-hypergraphs on $i$ vertices is bounded above by $c^{i^{a-2}}$.
Hence we can bound the above sum by 

\[
\begin{array}{rl}
\sum_{i=1}^{\rammkmc} c^{i^{a-2}}i^{a-2}& \le \rammkmc 2^{{\rammkmc}^{a-2}}{\rammkmc}^{a-2} \cr
                                        & \le \rammkmc^{a-1}  2^{{\rammkmc}^{a-2}}
\end{array}
\]

\end{proof}

\begin{theorem}\label{th:cfsaold}
For all $a\ge 3$, for all $k\ge 3$
$$\rakt \le 2^{\rammkmt^{a-1} 2^{{\rammkmt}^{a-2}}}.$$
\end{theorem}

\begin{proof}

Let $n$  be a number to be determined.
Let $COL$ be a 2-coloring of $\binom{[n]}{a}$.

We define a finite sequence of vertices $x_1,x_2,\ldots, x_L$ where
we will bound $L$ later.
For every $1\le i\le L$ we will also define $G_i$, 
a 2-colored $(a-2)$-hypergraph. We will represent $G_i$ as a subset of $\binom{\nat}{a-2}\times \{\REDns,\BLUEns\}$.
For example, if $a=6$, $G_i$ could be 

$$\{ (\{1,2,4,5\},\REDns), (\{1,3,4,9\},\BLUEns), (\{4,5,6,10\},\REDns) \}.$$

The notation $G_i = G_i \cup \{ \{12,13,19,99\},\REDns)) \}$ means that we add the edge 
$\{12,13,19,99\}$ to $G_i$ and color it $\REDns$ in $G_i$.

The construction will stop when one of the $G_i$ has 
a $(a-2)$-homogenous set of size $k-1$.
We will later show that this must happen.

\begin{definition}\label{de:agreea}
Let $G_{i_1},G_{i_2}$ be 2-colored $(a-2)$-hypergraphs. 
Let $J\in\binom{\nat}{a-2}$.
\begin{enumerate}
\item
$G_{i_1}$ and $G_{i_2}$ {\it agree on $J$} if
either 
(1) $G_{i_1}$ and $G_{i_2}$ both have edge $J$ and color it the same or
(2)  neither $G_{i_1}$ nor $G_{i_2}$ has edge $J$.
\item
$G_{i_1}$ and $G_{i_2}$ {\it agree on $\{1,\ldots,j\}$} if $G_{i_1}$ and $G_{i_2}$
agree on all of the edges in $\binom{[j]}{a-2}$.
\item
$G_{i_1}$ and $G_{i_2}$ {\it disagree on $J$} if
either (1) $G_{i_1}$ and $G_{i_2}$ both have edge $J$ and color it differently or (2)
one of them has edge $J$ but the other one does not.
\end{enumerate}
\end{definition}

\noindent
{\bf CONSTRUCTION}

\[
\begin{array}{rl}
x_1 =& 1 \cr
x_2 =& 2   \cr
\vdots = & \vdots \cr
x_{a-1}=& a-1 \cr
G_1= & \es \cr
G_2= & \es \cr
\vdots & \vdots \cr
G_{a-2} = &\es \cr
V_{a-2}= & [n]-\{x_1,\ldots,x_{a-1} \}. \hbox{ We start indexing here for convenience.} \cr
COL^*(y)= & COL(x_1,x_2,\ldots,x_{a-1},y) \hbox{ for all } y\in V_{a-2}\cr
V_{a-1} = & \hbox{ the largest $(a-2)$-homogeneous set for $COL^*$ }\cr
G_{a-1} = & ( \{ 1,\ldots,a-2\} ,\hbox{ the color of $V_{a-1}$) } \cr
\end{array}
\]

The $G_i$'s will be 2-colored $(a-2)$-hypergraphs.

\noindent
KEY: for all $y\in V_{a-1}$, $COL(x_1,\ldots,x_{a-1},y)$ is the color of 
$\{1,\ldots,a-2\}$ in $G_{a-1}$.

Let $i\ge a-1$, and assume that 
$V_{i-1}$, 
$x_1,\ldots,x_{i-1}$, and
$G_1,\ldots,G_{i-1}$ are defined.
If $G_{i-1}$ has an $(a-2)$-homogenous set of size $k-1$ then stop (yes $k-1$- this is not a typo).  Otherwise proceed.

\[
\begin{array}{rl}
G_{i} = & \es \hbox{ (This will change.) } \cr
x_{i} = & \hbox{ the least element of $V_{i-1}$ } \cr
V_{i} =&  V_{i-1} - \{x_{i}\} \hbox{ (We will change this set without changing its name.) }\cr
\end{array}
\]

We will add colored $(a-2)$-edges to $G_i$.
We will also define smaller and smaller sets $V_{i}$.
We will keep the variable name $V_{i}$ throughout.

\bigskip

In the next step we will, for all $J\in \binom{[i-1]}{a-2}$, consider
adding $J$ to $G_i$. The order in which we consider the $J$ matters.
Assume the order first considers each edge whose maximum entry is $a-2$,
then each edges with maximum entry is $a-1$, etc, until the maximum entry is $i-1$.

\noindent
For $J\in \binom{[i-1]}{a-2}$
\begin{enumerate}
\item
If for every $j\in J$, $G_j$ and $G_{i}$ agree on $\{1,\ldots,j-1\}$ then proceed, otherwise 
go to the next $J$.
(Note that when edge $J$ is being considered all of the edges $J'$ with $\max(J')<\max(J)$
have already been decided upon. Hence if $J$ becomes an edge of $G_i$ then it will
always be the case that, for every $j\in J$, $G_j$ and $G_i$ agree on $\{1,\ldots,j-1\}$.
\item
$COL^*: V_{i}\into \{\REDns,\BLUEns\}$ is defined by $COL^*(y) = COL(J\cup\{x_{i},y\})$.
\item
$V_{i}$ is the largest 1-homogeneous set for $COL^*$.
Note that $|V_{i}|$ decreases by at most half.
\item
$G_i = G_i \cup \{(J, \hbox{ the color of $V_i$ })\}$
\end{enumerate}

\noindent
KEY: Let $A\in \binom{[i-1]}{a-2}$ and $b>\max(A)$ such that
$A$ is an $(a-2)$-edge of $G_{i}$.
Let $c_A$ be such that $(A,c_A)\in G_{i}$.
For all $y\in V_i$, $COL(A\cup\{x_b,y\})=c_A$.

\noindent
{\bf END OF CONSTRUCTION}

When we derive upper bounds on $n$ we will show that the construction ends.  For now assume the construction ends.

When the construction ends we have a $G_L$ that
has a $(a-2)$-homogenous set of size $k-1$.
We assume the color is $\REDns$.
Let $\{i_1 < i_2 <\cdots < i_{k-1} \}$ be the $(a-2$)-homogenous set.
Define $i_k=L$.
We show that 
$$H= \{ x_{i_1},\ldots,x_{i_k} \}$$
is a $a$-homogenous set with respect to the original coloring $COL$.
For notational convenience we show 
that $COL(x_{i_1},\ldots,,x_{i_a})=\REDns$.
The proof for any $a$-set of $H$ is similar, 
even for the case where the last point is $x_L$.

Look at $G_{i_{a-1}}$. Since $i_{a-1}$ is a vertex  in $G_L$ 
we know that
$G_{i_{a-1}}$ and $G_L$ agree on \hbox{$\{1,\ldots,i_{a-1}-1\}$.}
Since $(i_{a-1},\REDns)\in G_L$ and $i_1,\ldots,i_{a-2} \le i_2-1$, 
$(\{i_1,\ldots,i_{a-2}\},\REDns)\in G_{i_{a-1}}$.
Hence, for all $y\in V_{i_{a-1}}$, $COL(x_{i_1},\ldots,x_{i_{a-2}},x_{i_{a-1}},y)=\REDns$.
In particular $COL(x_{i_1},\ldots,x_{i_a})=\REDns$.

We now establish bounds on $n$.

\begin{definition}
Let $G$ be a 2-colored ($a-2$)-hypergraph on vertex set $V=\{L_1<\cdots<L_m\}$ and edge set $E$.
Define $\squash(G)$ to be $G'=(V',E')$, the following 2-colored ($a-2)$-hypergraph:
\begin{itemize}
\item
The vertex sets $V'=\{1,\ldots,m\}$.
\item
For each edges $\{L_{i_1},\ldots,L_{i_{a-2}}\}$  in $E$
the edge $\{i_1,\ldots,i_{a-2}\}$ is in $E'$.
\item
The color of $\{i_1,\ldots,i_{a-2}\}$ in $G'$.
is the color of 
$\{L_{i_1},\ldots,L_{i_{a-2}}\}$ in $G$.
\end{itemize}
\end{definition}

\noindent
{\bf Claim 1:} For all $a-1\le i_1<i_2$, $\squash(G_{i_1})\ne\squash(G_{i_2})$.

\noindent
{\bf Proof of Claim 1:}
Assume, by way of contradiction, that $a-1\le i_1<i_2$ and $\squash(G_{i_1})=\squash(G_{i_2})$.
Let $G_{i_1}$ have vertex set $U_1$ and 
let $f_1$ be the isomorphism that maps $U_1$ to the vertex set of 
$\squash(G_{i_1}$).
Note $f_1$ is order preserving and,
if $f_1$ is applied to a number not in $U_1$, then
the result is undefined.
Define $U_2$ and $f_2$ for $G_{i_2}$ similarly.

We will prove that, for all $1\le j\le i_1-1$, 
(1) $f_1$ and $f_2$ agree on $\{1,\ldots,j\}$,
(2) $G_{i_1}$ and $G_{i_2}$ agree on $\{1,\ldots,j\}$.
The proof will be by induction on $j$.

\noindent
{\bf Base Case: $j\in\{1,2,\ldots,a-2\}$}. 
Since $a-1\le i_1,i_2$
the edge $E=\{1,2,\ldots,a-2\}$ is in both $G_{i_1}$ and $G_{i_2}$; therefore,
$f_1(1)=f_2(1)$, $\ldots$, $f_1(a-2)=f_2(a-2)$.
If the color of $E$ is different in $G_{i_1}$ and $G_{i_2}$ then 
$\squash(G_{i_1})\ne \squash(G_{i_2})$.
Hence the color of $E$ is the same in both graphs.
Thus we have that $G_{i_1}$ and $G_{i_2}$ agree on $\{1,\ldots,a-2\}$.

\noindent
{\bf Induction Step:} Assume that $G_{i_1}$ and $G_{i_2}$ agree on $\{1,2,\ldots,j-1\}.$
Assume that $f_1$ and $f_2$ agree on $\{1,\ldots,j-1\}$.
We use these assumptions without stating them throughout.
Look at what happens when $G_{i_1}$ ($G_{i_2}$) has to decide what to do with $j$.

If $G_j$ and $G_{i_1}$ agree on \hbox{$\{1,\ldots,j-1\}$} then, since $j<i_1$,
$G_j$ also agrees with $G_{i_2}$ on $\{1,\ldots,j-1\}$.
Hence the edge $\{1,2,\ldots,a-3,j\}$ will be put into both
$G_{i_1}$ and $G_{i_2}$. Hence $j$ will be a vertex in both $G_{i_1}$ and $G_{i_2}$
so $f_1(j)=f_2(j)$. Let $E\in \binom{[j]}{a-2}$ such that $j\in E$.
If for every vertex $j'$ of $E$,
$G_{j'}$ and $G_{i_1}$ agree on $\{1,\ldots,j'-1\}$ then, since $j'<i_1$,
$G_{j'}$ also agrees with $G_{i_2}$ on $\{1,\ldots,j'-1\}$.
Hence $E$ will be in both $G_{i_1}$ and $G_{i_2}$.
Since $f_1$ and $f_2$ agree on $\{1,\ldots,j\}$
and $\squash(G_{i_1})=\squash(G_{i_2})$, $E$ must be the same
color in $G_{i_1}$ and $G_{i_2}$. Hence 
every edge put into $G_{i_1}$
in stage $j$ is also in $G_{i_2}$ and with the same color.
By a similar argument we can show that
every edge put into $G_{i_2}$
in stage $j$ is also in $G_{i_1}$ and with the same color.
Hence $G_{i_1}$ and $G_{i_2}$ agree on $\{1,\ldots,j\}$.

If $G_j$ does not agree with $G_{i_1}$ on \hbox{$\{1,\ldots,j-1\}$} then there must be an edge $E\in \binom{[j-1]}{a-2}$
such that $G_j$ and $G_{i_1}$ disagree on $E$.
Hence $G_j$ and $G_{i_2}$ disagree on $E$.
Thus $j$ will not be made a vertex of $G_{i_1}$ or $G_{i_2}$ ever.
Hence both $f_1(j)$ and $f_2(j)$ are undefined.
No new edges are added to $G_{i_1}$ or $G_{i_2}$ in stage $j$
hence, since $G_{i_1}$ and $G_{i_2}$ agree on $\{1,\ldots,j-1\}$
they agree on $\{1,\ldots,j\}$.

We now know that $G_{i_1}$ and $G_{i_2}$ agree on $\{1,\ldots,i_1-1\}$.
Note that $G_{i_1}$ only has vertices in $\{1,\ldots,i_1-1\}$.
Look at stage $i_1$ in the construction of $G_{i_2}$.
Since $G_{i_1}$ agrees with $G_{i_2}$ on $\{1,\ldots,i_1-1\}$,
$i_1$ will be a vertex of $G_{i_1}$.
At that point $G_{i_2}$ will have more vertices then $G_{i_1}$ hence
$\squash(G_{i_1})\ne\squash(G_{i_2})$. This is a contradiction.

\noindent
{\bf End of Proof of Claim 1}

\noindent
{\bf Claim 2:} All of the $G_i$ are complete ($a-2$)-hypergraphs.

\noindent
{\bf Proof of Claim 2:}

Let $i_1 < i_2 < \cdots < i_{a-2}$ be vertices of $G_i$.
We will show that $\{i_1,\ldots,i_{a-2}\}$ is an edge in $G_i$.

For all $1\le j\le a-2$, since $i_j$ is a vertex of $G_i$
we know that $G_{i_j}$ and $G_i$ agree on $\{1,\ldots,i_j-1\}$.
Hence, in stage $i$, this will be noted and
$\{i_1,\ldots,i_{a-2}\}$ will be added to $G_i$.

\noindent
{\bf End of Proof of Claim 2}

We now bound $L$, the length of the sequence.
The sequence 
$G_1,G_2,\ldots,$ will end when some $G_i$ has 
$\rammkmt$ vertices (since by the definition of $\rammkmt$
there will be a homogenous set of size ($k-1$) or earlier.
For all $i\ge a-2$ map $G_i$ to $\squash(G_i)$. This mapping is 1-1 by Claim~1.
Hence the length of the sequence is bounded by the 
$a-3$ plus  the
number of 2-colored $(a-2)$-hypergraphs on {\it an initial segment of}  $\{1,\ldots,\rammkmt\}$,
so $L\le a-3 + 2^0 + \cdots + 2^{\rammkmt} \le 2^{\rammkmt+1}+a-4$.
We have shown the construction terminates.

Strangely enough, this is not quite what we care about when we are bounding $n$.
We care about the number of {\it edges} in all of the $G_i$'s since each
edge at most halves the number of vertices.

By Lemma~\ref{le:hypera} the number of edges in all of the $G_i$'s is
bounded above by $${\rammkmt}^{a-1}2^{{\rammkmt}^{a-2}}.$$
Hence the number of times that the number of vertices are decreased by 
at most half is bounded by this same quantity.
Therefore it suffices to take $n = 2^{{\rammkmt}^{a-1}2^{{\rammkmt}^{a-2}}}.$
Hence
$$\rakt \le 2^{\rammkmt^{a-1} 2^{{\rammkmt}^{a-2}}}.$$
\end{proof}

\pagebreak

\end{document}